\documentclass[leqno,11pt]{article}
\usepackage{amsmath}
\usepackage{array}
\usepackage{enumerate}
\usepackage{graphicx}
\usepackage{amssymb}
\usepackage{amsthm}
\usepackage{xcolor}

\newtheorem{proposition}{Proposition}
\newtheorem{lemma}{Lemma}
\newtheorem{corollary}{Corollary}
\newtheorem{definition}{Definition}
\newtheorem{remark}{Remark}
\newtheorem{example}{Example}

\begin{document}

\title{A ``barbell'' in a central force field:\\
A case study in symmetry reduction}

\author{
J\"urgen Scheurle\\
Zentrum Mathematik, TU M\"unchen\\
Boltzmannstr.~3, 85747 Garching, Germany\\
\tt{scheurle@ma.tum.de}\\
    \\
 Sebastian Walcher\\
Lehrstuhl A f\"ur Mathematik, RWTH Aachen\\
52056 Aachen, Germany\\
\tt{walcher@matha.rwth-aachen.de}
}

\date{}

\maketitle

\begin{abstract} We present an application of a recently introduced variant of orbit space reduction for symmetric dynamical systems. This variant works with suitable localizations of the algebra of polynomial invariants of the group actions, and provides reduction to a variety that is embedded in a low-dimensional affine space, which makes efficient computations possible. As an example, we discuss the mechanical system  of a ``barbell'' in a central force field. \\
{\bf Key words:} Linear groups, invariant theory, Hamiltonian systems with symmetry, relative equilibria\\
{\bf MSC (2010):} 34C20, 13A50, 34C14, 37C80
\end{abstract}

\bigskip

\section{Introduction}
Vladimir I.~Arnold made fundamental contributions to various mathematical disciplines and opened up new perspectives in several fields such as geometric mechanics, differential equations and dynamical sytems. Among his most celebrated works are those on the Hamiltonian stuctures of the Euler equations for rigid bodies and in fluid dynamics \cite{ArnoldGren,ArnoldUsp}. Symmetry reduction played an important role in these and other works of Arnold. A general abstract framework was developed by Marsden and Weinstein a few years later; see \cite{MarWei,Marsden, MarRat}. In addition to his groundbreaking research, Arnold authored several influential monographs on topics in classical mechanics, differential equations and dynamical systems; we only mention \cite{ArnoldMech,ArnoldGM} here.\\

The present paper is concerned with symmetry reduction and its application to a particular Hamiltonian system from mechanics. The general guideline for such reductions can be found e.g. in Arnold \cite{ArnoldMech} (Appendix 5), Marsden and Weinstein \cite{MarWei}, Kummer \cite{Kummer}, and Cushman and Bates \cite{CuBa} (Ch.~VII). The aspect we want to emphasize here is an efficient computational (``algebraic'') reduction procedure via invariants.\\

The mechanical system we consider is a ``barbell'' (two mass points connected by a rigid link) subject to a central force field in two-dimensional space. This system admits a linear symmetry group and -- as is common knowledge (see for example Chossat \cite{Cho}) -- one may employ the polynomial invariants of this symmetry group to construct a reduced system defined on an affine algebraic variety. The problem with the given system -- as well as many others -- is the high dimension of the embedding space for the variety, which renders any practical work with the reduced system almost impossible. We circumvent this difficulty by passing to suitable localizations of the polynomial invariant algebra. This method was recently introduced in \cite{SchrWa}, building on work by Grosshans \cite{Gros}. Thus one obtains an efficient reduction, which, in particular, allows the discussion of relative equilibria and their stability properties.\\

The plan of the paper is as follows. In section 2 we review symmetry reduction and then outline reduction via localizations, for general as well as for Hamiltonian systems.  In section 3 we introduce the ``barbell''
system, its symmetry group (a representation of $SO(2)$), and compute its symmetry reduction with
respect to a suitable localization of the invariant algebra of the
symmetry group. This particular reduction derives from a systematic
application of the theory developed in \cite{SchrWa}. It naturally yields a reduced
system with a Poisson structure in $\mathbb R^5$ which admits a first
integral due to angular momentum conservation. In turn, this allows a
further reduction to a Hamiltonian system in $\mathbb R^4$, which may be seen to be
the best possible outcome. To provide evidence for the practical
advantage of reduction via localization, we
first discuss the complete dynamics of the barbell system in the case
of a harmonic attracting force field. Second, in section 4 we consider
relative equilibria of the system, as well as their (linearized)
stability properties, in much more general situations showing that the
behavior of the system is quite intricate. In section 5, a few remarks
conclude the paper.

\section{A review of symmetry reduction}
For the readers' convenience we recall here some familiar (and some perhaps non-familiar) facts concerning symmetry reduction of ordinary differential equations, and introduce some notation.
\subsection{The basics }
Consider an autonomous ordinary differential equation
\begin{equation}\label{ode}
\dot x=F(x)
\end{equation}
on some nonempty and open subset of $\mathbb R^n$, with $F$ smooth (as a matter of convenience).  A (local) symmetry of this differential equation is a (local) diffeomorphism $\Phi$ that sends solutions to solutions (respecting the time parameterization). A necessary and sufficient criterion for this property is that the identity
\[
D\Phi(x)F(x)=F(\Phi(x))
\]
holds for all $x$.
\begin{definition}
The {\em Lie derivative} of a smooth scalar valued function $\psi$ with respect to the vector field $F$ is defined by
\[
\psi\mapsto L_F(\psi); \quad L_F(\psi)(x):=D\psi(x)F(x).
\]
\end{definition}
The Lie derivative measures the rate of change of $\psi$ along solutions $z(t)$ of \eqref{ode}, as
\[
\frac{d}{dt}\phi(z(t))=L_F(\phi)(z(t)).
\]
This fact is used for symmetry reduction by  invariants, as outlined in the following points.
\begin{itemize}
\item If $\Phi$ is a local symmetry of \eqref{ode} and $\psi$ is an invariant of $\Phi$, i.e. $\psi\circ\Phi=\psi$, then $L_F(\psi)$ is also an invariant of $\Phi$.\\
To verify this, differentiate
\[
\psi(\Phi(x))=\psi(x)\Rightarrow D\psi(\Phi(x))D\Phi(x)=D\psi(x)
\]
and use the symmetry condition to obtain
\[
\begin{array}{rcl}
L_F(\psi)(\Phi(x))&=&D\psi(\Phi(x))F(\Phi(x))\\
  &=&D\psi(\Phi(x))D\Phi(x)D\Phi(x)^{-1}F(\Phi(x))=D\psi(x)F(x).
\end{array}
\]
\item Therefore, if $\cal G$ is a collection of local symmetries of \eqref{ode}, and $\psi_1,\ldots,\psi_m$ are common invariants of the elements of $\cal G$ then every $L_F(\psi_j)$ is also a common invariant, $1\leq j\leq m$. If, furthermore, every common invariant of $\cal G$ can be expressed as a (smooth) function of the $\psi_j$ then the identities
\[
L_F(\psi_j)=\gamma_j(\psi_1,\ldots,\psi_m),\quad 1\leq j\leq m
\]
imply that solutions of \eqref{ode} are mapped to solutions of 
\[
\dot y_j=\gamma_j(y_1,\ldots,y_m),\quad 1\leq j\leq m
\]
by the Hilbert map
\[
\Psi:=\begin{pmatrix}\psi_1\\ \vdots \\ \psi_m\end{pmatrix}.
\]
\item The local setting described above transfers to global actions of Lie groups on $\mathbb R^n$ or some submanifold, with some restrictions. Locally, the existence of $\psi_1,\ldots,\psi_m$ is guaranteed by Frobenius' theorem, near any point with maximal dimension of its group orbit. (Globally one will have singular reduction in general; see e.g. Field \cite{Fie} for actions of compact groups.) Note that the Frobenius argument is not constructive (from an ``algebraic'' perspective), since it relies on the implicit function theorem. 
\end{itemize}

\subsection{Construction of reduced equations}
We now restrict attention to the natural action of an algebraic subgroup $G$ of $GL(n,\mathbb R)$ on $\mathbb R^n$, and a polynomial vector field $F$ that is symmetric with respect to $G$ (thus $TF(x)=F(Tx)$ for all $T\in G$). These assumptions are not as restrictive as they may seem, in view of classical linearization theorems for certain group actions (see e.g. Bredon \cite{Bre}, Thm.~4.1, for the compact case, and  Kushnirenko \cite{Kus} for semisimple groups). Moreover one should note Schwarz's \cite{Schw} and Poenaru's \cite{Poe} theorems on invariant smooth functions, resp. smooth symmetric vector fields (see also Luna \cite{Lun}). In this scenario reduction by invariants is in principle a constructive matter whenever the polynomial invariant algebra is finitely generated. 
\begin{proposition}\label{classred} If the invariant algebra $\mathbb R\left[x_1,\ldots,x_n\right]^G$ admits the finite set $\psi_1,\ldots, \psi_m$ of generators then the Hilbert map $\Psi=\left(\psi_1,\ldots, \psi_m\right)^{\rm tr}$ sends the $G$-symmetric vector field $F$ to some polynomial vector field $P$ on $\mathbb R^m$. The equation $\dot y=P(y)$ admits as an invariant set the algebraic variety $Z$ which is defined as the Zariski closure of $\Psi(\mathbb R^n)$ and determined by the polynomial relations between $\psi_1,\ldots,\psi_m$.
\end{proposition}
This procedure is known as {\em orbit space reduction}; see \cite{Sche} and Chossat \cite{Cho}. Its main drawback is that for many interesting group actions the polynomial invariant algebra (while finitely generated) needs a large number of generators; hence one has reduction to the variety $Z$ which is embedded in some (necessarily) high dimensional ambient space. This fact makes practical work with the reduced system awkward and frequently impossible. To circumvent this dilemma, one may employ localizations of the invariant algebra to achieve reduction to a rational system with powers of a single polynomial as denominators. This approach, which builds on Grosshans \cite{Gros}, is presented in detail in \cite{SchrWa}. We state a version here that is most appropriate for the application we will discuss. For proofs see \cite{SchrWa} (and use \cite{GSW}, Cor.~2.7, with regard to the characterization of $v$).
\begin{proposition}\label{locred} Let $\mathbb K$ denote $\mathbb R$ or $\mathbb C$, and let $G$ be an algebraic subgroup of $GL(n,\mathbb K)$ which acts naturally on $\mathbb K^n$, with finitely generated invariant algebra.
 Denote by $Q$ the quotient field
of $\mathbb K\left[x_1,\ldots,x_n\right]^G$ and let $q$ be its transcendence degree over $\mathbb K$. (Thus $q$  is at most equal to $n-s$, with $s$ the generic orbit dimension of the group action.)
Then for any $v$ with trivial isotropy group $G_v$ there exist an integer $\ell$ with $q\leq \ell\leq q+1$ and $\psi,\psi_1,\ldots,\psi_\ell\in\mathbb K\left[x_1,\ldots,x_n\right]^G$ such that
$$\mathbb K\left[x_1,\ldots,x_n\right]^G\left[\frac1\psi\right]=\mathbb K\left[\psi_1,\ldots,\psi_\ell\right]\left[\frac1\psi\right]\quad\text{and}\quad\psi(v)\not=0.$$
In particular every polynomial invariant can be written as the quotient of some polynomial in $\psi_1,\ldots,\psi_\ell$ and some power of $\psi$, in a Zariski neighborhood of $v$.
\end{proposition}
\begin{remark}\label{locrem}
{\em 
\begin{itemize}
\item One may use this Proposition to construct Hilbert maps of the type
\[
x\mapsto\begin{pmatrix}\psi_1\\ \vdots \\ \psi_\ell\\ \psi\end{pmatrix}, \text{  resp.  }x\mapsto\begin{pmatrix}\psi_1\\ \vdots \\ \psi_\ell\end{pmatrix}
\]
to reduce symmetric systems; the latter version works whenever $\psi\in \mathbb K[\psi_1,\ldots,\psi_\ell]$.
\item Note that $\ell=q$ and $\psi\in \mathbb K[\psi_1,\ldots,\psi_\ell]$ is the best situation one can hope for, since the generic orbit dimension determines the dimension of the quotient modulo the group action (in whichever way the quotient is realized). This best possible case does occur for toral subgroups, as is indicated by the examples in \cite{SchrWa} and proven in general in R.~Schroe\-ders' dissertation \cite{SchrDiss}.
\end{itemize}
}
\end{remark}
\subsection{The Hamiltonian setting}
Reduction for a symmetric Hamiltonian system (provided that the symplectic structure, resp. the Poisson bracket, is compatible with the group action) will produce a Hamiltonian system. A proof is given, and the procedure
is described, in a precise step-by-step manner in the monograph \cite{CuBa} by Cushman and Bates; see in particular Ch.~VII. (As noted earlier, other relevant sources are Arnold \cite{ArnoldMech}, Kummer \cite{Kummer} and Marsden/Weinstein \cite{MarWei}.) Cushman and Bates discuss the global scenario, with a Lie group acting on a symplectic manifold, and rather weak assumptions (properness) concerning the group action. The technical difficulty is that the reduction is singular in general;  actual computations are also carried out with the help of invariants.\\
Given a polynomial or rational Hamiltonian system that is symmetric with respect to an algebraic group action, (i.e., its Hamiltonian function is group invariant), and assuming that the structure matrix of the Poisson bracket has polynomial or rational entries, a direct method to determine a reduced Hamiltonian system (together with the Poisson bracket induced by the reduction) by polynomial invariants was introduced in \cite{SSW} and applied to a class of examples. In most instances this method amounts to a convenient computational shortcut for certain cases of the general reduction procedure in \cite{CuBa}, but the approach in \cite{SSW} applies to a different class of groups in comparison to \cite{CuBa} (including some non-reductive ones). In the present work we will slightly modify the approach from \cite{SSW} to discuss the barbell as a mechanical system.

\section{The system and its symmetry reductions}
\subsection{The system}
We consider a ``barbell'' that  consists of two mass points (with positive masses $m_1,\,m_2$)  in a central force field in the plane $\mathbb R^2$; they are connected by a massless rigid link of length $\ell$. Denote by $x$ the position and by $y=\dot x$ the velocity of the first particle, and by $z$ resp. $w$ the position and velocity of the second particle. The force field is characterized by a (sufficiently smooth) function $U: (0,\,\infty)\to \mathbb R$, $r\mapsto U(r)$, such that $m_1U(x_1^2+x_2^2)$ is the potential energy of the first particle and  $m_2U(z_1^2+z_2^2)$ the potential energy of the second one. (Slightly abusing terminology, we will sometimes call $U$ the potential.) The Hamiltonian of the unconstrained system of two particles is given by
\begin{equation}\label{freeham}
H=\frac12 m_1\left<y,y\right>+\frac12 m_2\left<w,w\right>+m_1U(\left<x,x\right>)+m_2U(\left<z,z\right>)
\end{equation}
with the standard scalar product $\left<\cdot,\cdot\right>$ on $\mathbb R^2$.
The canonical Poisson bracket on the tangent bundle $T(\mathbb R^2\times \mathbb R^2)\cong \mathbb R^8$ is given by
\begin{equation}\label{PBstan}
\begin{array}{rcl}
\{f,g\}&=&\frac1{m_1}\sum_i\left(\frac{\partial f}{\partial x_i}\frac{\partial g}{\partial y_i}-\frac{\partial f}{\partial y_i}\frac{\partial g}{\partial x_i}\right)\\
   &+&\frac1{m_2}\sum_i\left(\frac{\partial f}{\partial z_i}\frac{\partial g}{\partial w_i}-\frac{\partial f}{\partial w_i}\frac{\partial g}{\partial z_i}\right).\\
\end{array}
\end{equation}
By scaling we may assume that $\ell=1$, and we will do so from now on. The constraints are then described by
\begin{equation}\label{constr}
\begin{array}{rcccl}
c_1&:=& \left<x-z,x-z\right>-1&=&0;\\
c_2&:=&\left<x-z,y-w\right>&=&0.
\end{array}
\end{equation}
For the following discussion, a change of coordinates will sometimes be convenient; therefore we  introduce
\begin{equation}\label{uvcoords}
u:=x-z,\quad v:=y-w.
\end{equation}
Note that $c_1 =\left<u,u\right>-1$ and $c_2=\left<u,v\right>$ admit particularly simple expressions then.\\
To work out the equations of motion for the constrained system, we follow the general procedure in Cushman and Bates \cite{CuBa} to determine the Poisson-Dirac brackets on the constraint manifold.
With
\begin{equation}\label{Mdef}
\begin{array}{rcl}
\{c_1,c_2\}&=&\left(\frac1{m_1}+\frac1{m_2}\right)\cdot2\left<x-z,x-z\right>=\frac{2(m_1+m_2)}{m_1m_2}\\
\{c_1,H\}&=&  0;\\
\{c_2,H\}&=& \left<y-w,y-w\right>-2\left<x-z,U^\prime(\left<x,x\right>)x-U^\prime(\left<z,z\right>)z\right>
\end{array}
\end{equation}
and introducing the abbreviations
\[
\begin{array}{rcl}
M&:=& \frac{m_1m_2}{2(m_1+m_2}\\
A&:=& \left<y-w,y-w\right>-2\left<x-z,U^\prime(\left<x,x\right>)x-U^\prime(\left<z,z\right>)z\right>
\end{array}
\]
one obtains the matrix
\[
C:=\begin{pmatrix} 0&\{c_1,c_2\}\\ \{c_2,c_1\}&0\end{pmatrix}^{-1}=M\cdot\begin{pmatrix}0&-1\\ 1&0\end{pmatrix}
\]
on the constraint manifold defined by $c_1=c_2=0$ (cf.~Cushman and Bates \cite{CuBa}, eq. (36) on p.~302), which in turn gives rise to the Poisson-Dirac bracket
\begin{equation}\label{PBDir}
\{f,g\}^*:=\{f,g\}-\left(\{f,c_1\},\,\{f,c_2\}\right)\cdot C\cdot\begin{pmatrix}\{c_1,g\}\\ \{c_2,g\}\end{pmatrix}
\end{equation}
The time evolution of any function $q$ along solutions of the constrained system is then given by
\begin{equation}\label{qchange}
\dot q=\{q,H\}^*=\{q,H\}+M\cdot\{q,c_1\}\cdot A,
\end{equation}
in other words by the Lie derivative of $q$ with respect to the Hamiltonian vector field of $H$.
In particular, one obtains the equations of motion, which we write down component-wise for $u,v,z$ and $w$:
\begin{equation}\label{eqmotion}
\begin{array}{rcccl}
\dot u_i&=&\{u_i,H\}^*&=& v_i\\
\dot v_i&=&\{v_i,H\}^*&=& -2U^\prime(\left<z+u,z+u\right>)\cdot (z_i+u_i) +2U^\prime(\left<z,z\right>)\cdot z_i-\widetilde A\cdot u_i\\
\dot z_i&=&\{z_i,H\}^*&=& w_i\\
\dot w_i&=&\{w_i,H\}^*&=& -2U^\prime(\left<z,z\right>)\cdot z_i+\frac{m_1}{m_1+m_2}\cdot\widetilde A\cdot u_i\\
\end{array}
\end{equation}
Here $i\in\{1,2\}$ in each case, and
\begin{equation}\label{Adef1}
\widetilde A:= \left<v,v\right>-2\left<u,U^\prime(\left<u+z,u+z\right>)\cdot (u+z)-U^\prime(\left<z,z\right>)\cdot z\right>.
\end{equation}
One should note here that there exists abundant literature on dynamics and symmetry reduction of rigid bodies in (three dimensional) gravitational force fields, which typically is based on physical insight and geometric considerations in the spirit of Arnold, Marsden and others. We mention only Wang et al. \cite{WKM} as one representative of such work. In contrast, we present an approach that might be called ``algebraic'' and is amenable to algorithmic methods. Moreover it will prove to be well-suited for explicit computations.
\subsection{The symmetry group and its  invariants}
System \eqref{eqmotion} admits the representation of the planar rotation group on $\mathbb R^8$, given by
\[
\widetilde R:=\begin{pmatrix}R&0&0&0\\0&R&0&0\\0&0&R&0\\0&0&0&R\end{pmatrix},\quad R=\begin{pmatrix}\cos\theta&-\sin\theta\\ \sin\theta&\cos\theta\end{pmatrix}
\]
as a symmetry group $G$. (This matrix representation holds with respect to $x,y,z,w$ as well as $u,v,z,w$.) To compute the symmetry reduction, we first list a generator system for the polynomial invariants.
\begin{lemma}\label{rholem}
The polynomial  invariant algebra of $G$ is generated by the sixteen polynomials
\begin{equation}
\begin{array}{ccc}
\rho_1=u_1^2+u_2^2;&\rho_2=v_1^2+v_2^2;&\rho_3=z_1^2+z_2^2;\\
\rho_4=w_1^2+w_2^2;&\rho_5=u_1v_1+u_2v_2;& \rho_6=u_1z_1+u_2z_2;\\
\rho_7=u_1w_1+u_2w_2;&\rho_8=v_1z_1+v_2z_2;&\rho_9=v_1w_1+v_2w_2;\\
\rho_{10}=z_1w_1+z_2w_2;&\rho_{11}=v_1u_2-u_1v_2;&\rho_{12}=z_1u_2-u_1z_2;\\
\rho_{13}=w_1u_2-u_1w_2;&\rho_{14}=z_1v_2-v_1z_2;& \\
\rho_{15}=w_1v_2-v_1w_2;&\rho_{16}=z_1w_2-z_2w_1 .&\\
\end{array}
\end{equation}
\end{lemma}
\begin{proof}[Sketch of proof]
Diagonalizing $\widetilde R$ one obtains a diagonal matrix with entries $a:=\exp(i\theta)$ and $a^{-1}$. In eigencoordinates a generator system of the invariant algebra is given by quadratic monomials; see \cite{SchrWa}, Example 1 and Example 4 for the explicit expressions. There remains to take real and imaginary parts.
\end{proof}
Note that all the invariants in Lemma \ref{rholem} are expressible as scalar products or determinants; for instance $\rho_7=\left<u,w\right>$ and $\rho_{11}=\det(v,u)$. \\
One should emphasize that this is a smallest set of polynomial generators (which satisfy a number of relations not listed here). By Proposition \ref{classred} one may reduce system \eqref{eqmotion} via the Hilbert map constructed from the $\rho_i$; this is indeed a reduction to a seven dimensional subvariety of $\mathbb R^{16}$ and there is an induced Poisson bracket on this variety for which the reduced system is Hamiltonian (see \cite{SSW}, in particular Prop.~A.5 in the Appendix). Since practical computations with this reduction do not seem feasible, we take a different approach by Proposition \ref{locred} here.
\begin{lemma}\label{etalem}
With the subset of the generator system defined by
\begin{equation}
\eta_1:=\rho_1;\,\eta_2:=\rho_5;\,\eta_3:=\rho_{11};\, \eta_4:=\rho_6;\,\eta_5:=\rho_{12};\eta_6:=\rho_7;\,\eta_7:=\rho_{13}
\end{equation}
the following relations hold:
\[
\begin{array}{ccc}
\rho_2=(\eta_2^2+\eta_3^2)/\eta_1;&\rho_3=(\eta_4^2+\eta_5^2)/\eta_1;&\rho_4=(\eta_6^2+\eta_7^2)/\eta_1;\\
\rho_8=(\eta_2\eta_4+\eta_3\eta_5)/\eta_1;&\rho_9=((\eta_2\eta_6+\eta_3\eta_7)/\eta_1;&\rho_{10}=(\eta_4\eta_6+\eta_5\eta_7)/\eta_1;\\
\rho_{14}=(\eta_2\eta_5-\eta_3\eta_4)/\eta_1;&\rho_{15}=(\eta_2\eta_7-\eta_3\eta_6)/\eta_1;&\rho_{16}=(\eta_4\eta_7-\eta_5\eta_6)/\eta_1.
\end{array}
\]
Therefore any polynomial invariant of $G$ may be expressed as a rational function in $\eta_1,\ldots,\eta_7$ with only powers of $\eta_1$ occurring in the denominator.
\end{lemma}
The verification of this lemma is straightforward. However, we emphasize that the choice of $\eta_1,\ldots,\eta_7$ is not a matter of chance but naturally follows from the theory developed in \cite{SchrWa}, Theorem 1 and Example 4. There are several possible localizations; the one in Lemma \ref{etalem} was chosen in view of the application to the barbell.
\subsection{Reduction}
Considering Lemma \ref{etalem}, Proposition \ref{locred} and Remark \ref{locrem} it is natural to introduce the Hilbert map
\begin{equation}
E:\, \mathbb R^8\to\mathbb R^7,\quad\begin{pmatrix} u\\v\\z\\w\end{pmatrix}\mapsto\begin{pmatrix}\eta_1(u,v,z,w)\\ \vdots \\ \eta_7(u,v,z.w)\end{pmatrix}.
\end{equation}
We denote the coordinates in $\mathbb R^7$ by $s_1,\ldots, s_7$.
 The actual reduction is as follows (no denominator occurs due to $\eta_1=1$ on the constraint manifold).
\begin{proposition}\label{redpropo}
\begin{enumerate}[(a)]
\item The map $E$ sends solutions of system \eqref{eqmotion} to solutions of an equation in $\mathbb R^7$ with $\dot s_1=0$, $\dot s_2=0$, and a remaining five dimensional  system
\begin{equation}\label{reqmotion}
\begin{array}{rcl}
\dot s_3&=& -2\left(U^\prime((1+s_4)^2+s_5^2)-U^\prime(s_4^2+s_5^2)\right)\cdot s_5\\
\dot s_4&=& s_3s_5+s_6\\
\dot s_5&=&-s_3s_4+s_7\\
\dot s_6&=& -\frac{2m_1}{m_1+m_2}\cdot U^\prime\left((1+s_4)^2+s_5^2\right)\cdot (1+s_4)\\ & &-\frac{2m_2}{m_1+m_2}\cdot U^\prime(s_4^2+s_5^2)\cdot s_4
  +\frac{m_1}{m_1+ m_2}s_3^2+s_3s_7\\
\dot s_7&=& -2U^\prime(s_4^2+s_5^2)\cdot s_5 -s_3s_6
\end{array}
\end{equation}
We will refer to \eqref{reqmotion} as {\em the reduced system}.
\item System \eqref{reqmotion} is Hamiltonian with respect to an induced Poisson bracket $\{\cdot,\,\cdot\}^\prime$. Its structure matrix
\[
\begin{pmatrix} \{\cdot,\,\cdot\}^\prime&| & s_3& s_4&s_5&s_6&s_7\\
                                                           -   & -  &- &-&-&-&-\\
s_3&|& 0 &-\frac{s_5}{2M}& \frac1{m_2}+\frac{s_4}{2M}&-\frac{s_3}{m_2}-\frac{s_7}{2M}& \frac{s_6}{2M}\\
s_4&|& -\frac{s_5}{2M} & 0 & 0 &\frac{1}{m_1+m_2}& -\frac{s_5}{m_2}\\
s_5&|&-\frac1{m_2}-\frac{s_4}{2M}&0&0&0& \frac{1+s_4}{m_2}\\
s_6&| & \frac{s_3}{m_2}+\frac{s_7}{2M}&-\frac{1}{m_1+m_2}&0&0&-\frac{2Ms_3}{m_2^2}\\
s_7&|&  -\frac{s_6}{2M}&\frac{s_5}{m_2}&-\frac{1+s_4}{m_2}& \frac{2Ms_3}{m_2^2} & 0
\end{pmatrix}
\]
has constant rank $4$. The reduced Hamiltonian function is
\[
\begin{array}{rcl}
h&:=&\frac{m_1}2\left(s_3^2+2s_3s_7+s_6^2+s_7^2\right)+\frac{m_2}2\left(s_6^2+s_7^2\right)\\
   & & \quad m_1U\left((1+s_4)^2+s_5^2\right)+m_2U\left(s_4^2+s_5^2\right).
\end{array}
\]
\item System \eqref{reqmotion} admits the first integral
\[
j:=m_1\left(s_3+s_7+s_3s_4\right) + (m_1+m_2)\left(s_4s_7-s_5s_6\right).
\]
\end{enumerate}
\end{proposition}
\begin{proof} We just sketch some  arguments of the proof, omitting straightforward (but lengthy) calculations. To prove part (a), use \eqref{qchange} and Proposition \ref{locred}, and re-express any invariant polynomial via $\eta_1,\ldots,\eta_7$ by Lemma \ref{etalem}. For instance, one computes
\[
\begin{array}{rcl}
\{\eta_4,H\}^*&=&u_1\{z_1,H\}^*+z_1\{u_1,H\}^*+u_2\{z_2,H\}^*+z_2\{u_2,H\}^*\\
               &=&u_1w_1+z_1v_1+u_2w_2+z_2v_2\\
               &=& \rho_7+\rho_8=\eta_6+\eta_2\eta_4+\eta_3\eta_5;
\end{array}
\]
and then recalls that $\eta_2=0$ on the constraint manifold.\\
The proof of part (b) is based on the existence of an induced Poisson bracket which is characterized by the identity
\[
\{f\circ E,\,g\circ E\}^*=\{f,g\}^\prime\circ E
\]
for polynomial functions on $\mathbb R^5$ (see Proposition A5 and its proof in \cite{SSW}; the argument also applies to the given situation, since the image of $E$ -- by the rank of its Jacobian --  is Zariski dense in $\mathbb R^7$.) The rest follows from straightforward computations again; for instance the equality
\[
\begin{array}{rcl}
\{\eta_3,\eta_4\}^*&=& \{\rho_{11},\rho_6\}^*\\
  &=& -\rho_{12}/(2M)= -\eta_5/(2M)
\end{array}
\]
implies that
\[
\{s_3,s_4\}^\prime=-s_5/(2M).
\]
The lower right $4\times 4$ minor of the structure matrix is equal to 
\[
(1+s_4)^2/(m_1^2(m_1+m_2)^2),
\]
 and when $s_4=-1$ then the upper left $4\times 4$ minor equals 
$1/(m_1^2(m_1+m_2)^2)$; therefore the structure matrix has constant rank four.
The reduced Hamiltonian $h$ is obtained by rewriting $H$ as a function of the $\eta_i$.\\
Finally, part (c) is a consequence of the $G$-symmetry of \eqref{eqmotion}: The infinitesimal generator of $G$ is 
\[
{\rm diag}(B,B,B,B); \quad B:=\begin{pmatrix}0&-1\\1&0\end{pmatrix}
\]
and the corresponding vector field (which commutes with the right hand side of \eqref{eqmotion}) is Hamiltonian, with the angular momentum
\[
\begin{array}{rcl}
J&=&m_1\left(y_1x_2-x_1y_2\right)+m_2\left(w_1z_2-z_1w_2\right)\\
  &=&m_1\left(\eta_3+\eta_7+\eta_3\eta_4\right)+(m_1+m_2)\left(\eta_4\eta_7-\eta_5\eta_6\right)
\end{array}
\]
as Hamilton function. Therefore $J$ is also a first integral of \eqref{eqmotion}.
Defining $j$ by $j\circ E=J$, one has
\[
\{j,h\}^\prime\circ E=\{J,H\}^*=0
\]
by \cite{SSW}, and the assertion follows.
\end{proof}

\begin{remark}{\em The first integral $j$ yields a further reduction of \eqref{reqmotion} to any level set $j=j_0={\rm const.}$, thus (generically) to a hypersurface in $\mathbb R^5$. By a further transformation, this system can be embedded in $\mathbb R^4$: Defining 
\[
\widetilde s_3:=s_3+\frac{m_1+m_2}{m_1}s_7,
\]
one sees that
\[
j=\widetilde s_3-\frac{m_2}{m_1}s_7+m_1\widetilde s_3s_4-(m_1+m_2)s_5s_6
\]
and therefore
\[
\frac{m_2}{m_1}s_7=\widetilde s_3+m_1\widetilde s_3s_4-(m_1+m_2)s_5s_6-j_0
\]
on the level set. This yields a differential equation system for $\widetilde s_3,\,s_4,s_5$ and $s_6$. We will not use this further reduction in most of the present paper, but it turns out to be useful in subsection \ref{subseqharm} below.\\
The underlying reason for the introduction of $\widetilde s_3$ becomes transparent from carrying out a ``completion of squares'' for the quadratic form 
\[
m_1 s_3s_4 + (m_1+m_2)\left(s_4s_7-s_5s_6\right)
\]
occurring in $j$.}
\end{remark}
We close this subsection with some remarks on the image of the constraint manifold under
\[
\widetilde E:\,\begin{pmatrix}u\\v\\z\\w\end{pmatrix}\mapsto\begin{pmatrix}\eta_3\\ \vdots \\ \eta_7\end{pmatrix}.
\]
By general arguments, this is a semialgebraic subset of $\mathbb R^5$ (Tarski-Seidenberg) which contains a nonempty open subset of $\mathbb R^5$ (by the generic rank of the Jacobian). One can use part of Procesi and Schwarz \cite{ProcesiSchwarz} to find inequalities which must be satisfied by the image: The matrix
\[
M(x):=DE(x)\cdot \left(DE(x)\right)^{\rm tr}
\]
has $G$-invariant entries (which can be written as polynomials in $\eta_3,\ldots,\eta_7$ on the constraint manifold). Moreover it is positive semidefinite by construction, hence its Hurwitz determinants are nonnegative. It is not clear, however, whether these inequalities also suffice; the proof in \cite{ProcesiSchwarz} is not directly applicable. Since we will not require precise information about the image in the following, we will not discuss this any further.
\subsection{Constrained harmonic motion}\label{subseqharm}
As a simple (but nontrivial) application we consider the case of harmonic potential energy, thus
\[
U^\prime=:\gamma>0.
\]
We first note that the case of the unconstrained system of the two particles is then straightforward: All nonconstant solutions of the equation of motion are periodic, with period $2\pi/\sqrt\gamma$. But for the constrained system
the equations of motion are nonlinear, with reduced system 
\[
\begin{array}{rcl}
\dot s_3&=& 0\\
\dot s_4&=& s_3s_5+s_6\\
\dot s_5&=&-s_3s_4+s_7\\
\dot s_6&=& -\frac{2\gamma m_1}{m_1+m_2}\cdot (1+s_4)-\frac{2\gamma m_2}{m_1+m_2}\cdot s_4 +\frac{m_1}{m_1+ m_2}s_3^2+s_3s_7\\
\dot s_7&=& -2\gamma\cdot s_5 -s_3s_6
\end{array}
\]
The first equation shows that $\eta_3=\det(v,u)=:\sigma$ is constant in \eqref{eqmotion} (note the relation to $j$). Together with $\eta_2=\left<u,v\right>=0$ and $\left<u,u\right>=1$, this implies that
\[
v=\sigma\begin{pmatrix}u_2\\-u_1\end{pmatrix}.
\]
Upon substitution of $s_3=\sigma$ in the reduced equation, there remains the linear system
\begin{equation}\label{reqharm}
\frac{d}{dt}\begin{pmatrix}s_4\\s_5\\s_6\\s_7\end{pmatrix}=\begin{pmatrix}0&\sigma&1&0\\ -\sigma&0&0&1\\
-2\gamma&0&0&\sigma\\
0&-2\gamma&-\sigma&0\end{pmatrix}\begin{pmatrix}s_4\\s_5\\s_6\\s_7\end{pmatrix}+\begin{pmatrix}0\\0\\\frac{(\sigma^2-2\gamma)m_1}{m_1+m_2}\\0\end{pmatrix}.
\end{equation}
This is a two-degree-of-freedom Hamiltonian system, by Proposition \ref{redpropo}. The eigenvalues of the matrix, thus
\[
\pm i\cdot(\sigma \pm \sqrt{2\gamma})
\]
are pairwise distinct except for the cases $\sigma^2=2\gamma$ and $\sigma=0$, respectively. 
The matrix is invertible whenever $\sigma^2-2\gamma\not=0$. In this case, system \eqref{reqharm} admits the unique stationary point
\[
\frac{m_1}{m_1+m_2}\begin{pmatrix}1\\0\\0\\ \sigma\end{pmatrix}.
\]
The physical interpretation of the stationary point (a relative equilibrium of the original system \eqref{eqmotion}) is straightforward: We have 
\[
\eta_4=\left<u,z\right>=\frac{m_1}{m_1+m_2} \text{  and  } \eta_5= \det(z,u)=0
\]
from the first and second entry, hence
\[
z=\frac{m_1}{m_1+m_2}u,\quad x=u+z.
\]
The remaining two conditions then yield
\[
w=\frac{m_1}{m_1+m_2}v=\frac{\sigma m_1}{m_1+m_2}\begin{pmatrix}u_2\\-u_1\end{pmatrix}.
\]
Thus, the orientation of the barbell is radial (both particles are on a line through the origin), and it rotates around the center with constant angular velocity. (See also Remark \ref{s5rem} below.)\\
Let us next consider non-stationary solutions of \eqref{reqharm} in the non-exceptional cases with $\sigma\not=0$ and $\sigma^2-2\gamma\not=0$. 
As is well-known, the eigenvalue ratio
\[
\omega:=\frac{\sigma -\sqrt{2\gamma}}{\sigma +\sqrt{2\gamma}}
\]
determines the dynamics:  In the non-resonant cases ($\omega\not\in\mathbb Q$) every nonconstant solution of \eqref{reqharm} is quasiperiodic, i.e. dense on a two-dimensional torus. But for rational $\omega$ all solutions are periodic (with their trajectories homeomorphic to circles). Hence, arbitrarily small changes in $\sigma$ will change the qualitative behavior substantially.\\
There remain the exceptional cases. When $\sigma=0$ then the matrix is semisimple, and all nonconstant solutions are periodic with period $2\pi/\sqrt{2\gamma}$. Finally, when $\sigma^2-2\gamma=0$ then equation \eqref{reqharm} is homogeneous, with matrix of rank two. There is a two-dimensional subspace of stationary points, and every nonstationary solution is periodic with period $\pi/\sqrt{2\gamma}$. 
\section{Relative equilibria}
Subsection \ref{subseqharm} already gave an indication that the reduction  of system \eqref{eqmotion} via  Proposition \ref{redpropo} is convenient for actual computations, and we will further illustrate this fact in the discussion of relative equilibria.
Relative equilibria of system \eqref{eqmotion} are equilibria of the reduced differential equation \eqref{reqmotion}, thus they solve the system of nonlinear ``algebraic'' equations
\begin{equation}\label{requilib}
\begin{array}{rcl}
0&=& -2\left(U^\prime((1+s_4)^2+s_5^2)-U^\prime(s_4^2+s_5^2)\right)\cdot s_5\\
0&=& s_3s_5+s_6\\
0&=&-s_3s_4+s_7\\
0&=& -\frac{2m_1}{m_1+m_2}\cdot U^\prime\left((1+s_4)^2+s_5^2\right)\cdot (1+s_4)\\ & & -\frac{2m_2}{m_1+m_2}\cdot U^\prime(s_4^2+s_5^2)\cdot s_4
  +\frac{m_1}{m_1+ m_2}s_3^2+s_3s_7\\
0&=& -2U^\prime(s_4^2+s_5^2)\cdot s_5 -s_3s_6
\end{array}
\end{equation}
We will discuss these equilibria, assuming throughout that $U^\prime$ is not constant. The first equation of \eqref{requilib} gives rise to a natural distinction of cases. 
\subsection{First case: $s_5=0$}\label{subsecone}
\begin{proposition}\label{releqone}
The equilibria of \eqref{reqmotion} with $s_5=0$ and  $m_1+(m_1+m_2)s_4\not=0$ are characterized by the relations
\[
\begin{array}{rcl}
s_3^2&=&\frac{2m_1U^\prime\left((1+s_4)^2\right)(1+s_4)+2m_2U^\prime\left(s_4^2\right)s_4}{m_1+(m_1+m_2)s_4}\\
s_5&=&0\\
s_6&=&0\\
s_7&=&s_3s_4
\end{array}
\]
with $s_4$ running through all values such that the right hand side of the first equation is $\geq 0$.\\
Whenever $m_1\not=m_2$ and $U^\prime$ is strictly monotone then no equilibria with $s_5=0$ and  $m_1+(m_1+m_2)s_4=0$ exist.
\end{proposition}
\begin{proof}This follows from a straightforward evaluation of \eqref{requilib}: With $s_5=0$, the second equation immediately implies that $s_6=0$ (hence the last equation is automatically satisfied). Then one may substitute $s_7=s_3s_4$ in the fourth equation.
The relation for $s_3^2$ is actually 
\[
(m_1+(m_1+m_2)s_4)s_3^2=2m_1U^\prime\left((1+s_4)^2\right)(1+s_4)+2m_2U^\prime\left(s_4^2\right)s_4
\]
which may be restated as above whenever $m_1+(m_1+m_2)s_4\not=0$. In case $s_4=-m_1/(m_1+m_2)$ there remains
\[
0=2\frac{m_1m_2}{m_1+m_2}\left(U^\prime((\frac{m_2}{m_1+m_2})^2)-U^\prime((\frac{m_1}{m_1+m_2})^2)\right)
\]
which has no solution whenever $m_1\not=m_2$ and $U^\prime$ is strictly monotone.
\end{proof}
We will discuss the additional equilibria in case $m_1=m_2$ below in subsection \ref{emsubec}.
\begin{remark}\label{s5rem}
{\em 
The condition $\eta_5=0$ admits a natural physical interpretation: Since 
\[
\eta_5=\det(z,u)
\]
one sees that $\eta_5=0$  if and only if $u$ and $z$ are parallel, thus the barbell is positioned radially, on a line through the center. Given a suitable initial state (with the further conditions on $\eta_3,\eta_4$ and $\eta_5$ also satisfied), the arrangement rotates around the center, since its path is restricted to a $G$-orbit. Moreover, with $\eta_4=\left<u,z\right>$ one has then
\[
z=\eta_4 u\text{  and  }x=(\eta_4+1)u.
\]
Thus, for $\eta_4>0$ both particles are on a straight line through the center and at the same side of the center, with the particle of mass $m_2$ closer to the center; for $\eta_4<-1$ both masses are on the same side, with the particle of mass $m_1$ closer to the center, and for $-1<\eta_4<0$ the particles are positioned at different sides of the center. (The borderline cases $\eta_4\in\{0,1\}$ describe the setting when one particle lies in the center; this may or may not be permissible, depending on the potential.)}
\end{remark}
Considering (linear) stability properties, it is possible to compute the linearization at a stationary point of \eqref{reqmotion} and its characteristic polynomial, but there seems to be little information to be gleaned from this for general potentials. The characteristic polynomial has a root $0$, due to the existence of the first integral $j$. (Recall that we did not specify a level set for $j$ above.) The remaining eigenvalues then determine the (linear) orbital stability properties of the relative equilibria of system \eqref{eqmotion}. We will consider only a special potential here, which already exhibits rather intricate behavior.
\begin{example}{\em We consider gravitation in two dimensions, thus
\[
U^\prime(r)=\frac1{r}
\]
after suitable scaling. The first condition in Proposition \ref{releqone} then becomes
\[
s_3^2=\frac{2(s_4(m_1+m_2)+m_2)}{(1+s_4)s_4(s_4(m_1+m_2)+m_1)},
\]
hence the right hand side must be defined and nonnegative; this provides restrictions on $s_4$.
\begin{itemize}
\item In case $m_1>m_2$ relative equilibria exist for
\[
s_4\in(-\infty,-1)\cup(\frac{-m_1}{m_1+m_2}, \frac{-m_2}{m_1+m_2}]\cup(0,\infty).
\]
\item In case $m_1=m_2$ relative equilibria exist for
\[
s_4\in(-\infty,-1)\cup(0,\infty).
\]
\item  In case $m_1<m_2$ relative equilibria exist for
\[
s_4\in(-\infty,-1)\cup[\frac{-m_2}{m_1+m_2}, \frac{-m_1}{m_1+m_2})\cup(0,\infty).
\]
\end{itemize}
The characteristic polynomial at such a stationary point is given by
\[
\chi(t)=t\cdot(t^4+C_1t^2+C_2).
\]
Here
\[
C_1=\frac{N_1}{s_4^2(m_1+m_2)(1+s_4)^2(s_4(m_1+m_2)+m_1)}
\]
with
\[
\begin{array}{rcl}
N_1&=&(8m_1^2+16m_1m_2+8m_2^2)s_4^3+(14m_1^2+24m_1m_2+10m_2^2)s_4^2\\
&& +(8m_1^2+8m_1m_2+4m_2^2)s_4+2m_1^2,
\end{array}
\]
and
\[
C_2=\frac{N_2}{s_4^4(1+s_4)^4(m_1+m_2)(s_4(m_1+m_2)+m_1)^2}
\]
with
\[
\begin{array}{rcl}
N_2&=&16(m_1+m_2)^3s_4^6+(56m_1^3+152m_1^2m_2+136m_1
m_2^2+40m_2^3)s_4^5\\
&&+(72m_1^3+160m_1^2m_2+120m_1m_2^2+32m_2^3)s_4^4\\
&&+(40m_1^3+56m_1^2m_2+24m_1m_2^2+
8m_2^3)s_4^3\\
&&+(8m_1^3-12m_1^2m_2-20m_1m_2^2)s_4^2+(-16m_1^2m_2-8m_1m_2^2)s_4-4m_1^2m_2.
\end{array}
\]
Therefore linear stability of an equilibrium with given $s_4$ is determined by the roots of the quadratic polynomial
\[
\widehat\chi(\tau)=\tau^2+C_1\tau+C_2.
\]
The discriminant of $\widehat\chi$ equals
\[
D=\frac{4D_1\cdot (s_4(m_1+m_2)+m_1)^2}{s_4^4(1+s_4)^4(m_1+m_2)(s_4(m_1+m_2)+m_1)^2}
\]
with
\[
D_1=(9m_1^2+14m_1m_2+9m_2^2)s_4^2+(6m_1^2+14m_1m_2+12m_2^2)s_4+m_1^2+4m_1m_2
+4m_2^2.
\]
Since, in turn, the discriminant of $D_1$ as a polynomial in $s_4$ is equal to
\[
-32(m_1+m_2)^2m_1m_2<0,
\]
one sees that $D_1$ and $D$ (when defined) are both $\geq 0$, hence all roots of $\widehat \chi$ are real, and their signs determine stability. We give a brief discussion.
\begin{itemize}
\item  $s_4\in(0,\infty)$: Since $N_2=-4m_1^2m_2$ when $s_4=0$, we have $C_2<0$ for small $s_4>0$, and the system is unstable for those values of $s_4$. On the other hand, the asymptotic behavior of $C_1$ and $C_2$ as $s_4\to\infty$ implies that $C_1>0$ and $C_2>0$ for sufficiently large $s_4$. Since also $D> 0$, both roots of $\widehat \chi$ are negative, and we have linear stability. (Numerical examples indicate that there occurs precisely one change from instability to stability as $s_4$ grows; but this seems not easy to prove in general.)
\item $s_4\in(-\infty,-1)$: Since $(1+s_4)^4C_2\to-4m_1/(m_1+m_2)$ as $s_4\to -1$, one sees that $\widehat\chi$ has  a positive root for $s_4<-1$ but close to $-1$. As $s_4\to-\infty$ one obtains the existence of two negative roots by the same resoning as above. Hence there occurs a change of the stability properties as in the first case.
\item In case $m_1>m_2$ and $s_4\in(\frac{-m_1}{m_1+m_2}, \frac{-m_2}{m_1+m_2}]$ one finds that $C_2=0$ when $s_4=-1/2$, and a Taylor expansion shows that the sign of $C_2$ changes from $+$ to $-$ whenever $m_2<m_1<(3+2\sqrt2)m_2$, and from $-$ to $+$ whenever $m_1>(3+2\sqrt2)m_2$. These facts show that a change of the stability properties takes place, and they indicate that stability properties depend in a quite subtle manner on mass ratios. Similar observations apply to the case $m_1<m_2$. We will not discuss further details here.
\end{itemize}
}
\end{example}
\subsection{Second case: $s_5\not=0$}
A full discussion of the necessary condition
\[
U^\prime((1+s_4)^2+s_5^2)=U^\prime(s_4^2+s_5^2)
\]
would be quite intricate, but focussing attention on strictly monotone $U^\prime$ (which is a reasonable restriction) yields rather general results.
\begin{proposition}\label{releqtwo}
Assume that $U^\prime$ is either strictly increasing or strictly decreasing. Then the following hold:
\begin{enumerate}[(a)]
\item The equilibria of \eqref{reqmotion} with $s_5\not=0$ are characterized by the relations
\[
\begin{array}{rcl}
s_3^2&=&2U^\prime\left(\frac14+s_5^2\right)\\
s_4&=&-\frac12\\
s_6&=&-s_3s_5\\
s_7&=&-\frac12s_3
\end{array}
\]
with $s_5$ running through all values such that $U^\prime(\frac14+s_5^2)\geq 0$. In particular, for a repelling force with $U^\prime <0$ no such equilibria exist.
\item The characteristic polynomial at such a stationary point has the form 
\[
\chi(t)=t\cdot\left(t^4+C_1t^2+C_2\right)
\]
with
\[
\begin{array}{rcl}
C_1&=&U^{\prime\prime}(\frac14+s_5^2)\cdot(8s_5^2+1)+8U^\prime(\frac14+s_5^2)\\
C_2&=&U^{\prime\prime}(\frac14+s_5^2)^2\cdot(16s_5^4+4s_5^2)+32U^{\prime\prime}(\frac14+s_5^2)\cdot U^{\prime}(\frac14+s_5^2)\cdot s_5^2.
\end{array}
\]
Its linear stability properties are therefore determined by the zeros of the degree $2$ polynomial
\[
\widehat\chi =\tau^2+C_1\tau+C_2
\]
which are real since the discriminant of $\widehat\chi$ equals
\[
D:=\left(U^{\prime\prime}(\frac14+s_5^2)+8U^{\prime}(\frac14+s_5^2)\right)^2\geq 0.
\]
\end{enumerate}
\end{proposition}
\begin{proof}By strict monotonicity of $U^\prime$, the equation
\[ U^\prime((1+s_4)^2+s_5^2)=U^\prime(s_4^2+s_5^2)\] is satisfied if and only if $(1+s_4)^2+s_5^2=s_4^2+s_5^2$; equivalently $s_4=-\frac12$. The remaining relations are straightforward (substituting $s_6=-s_3s_5$ in the last equation with $s_5\not=0$). This shows part (a).\\
Part (b) is the result of a calculation starting from the Jacobian (upon substitution of $s_4=-1/2$)
\[
\begin{pmatrix}0&-4U^{\prime\prime}(\frac14+s_5^2)s_5&0&0&0\\
s_5&0&s_3&1&0\\
\frac12&-s_3&0&0&1\\\frac{(m_1+m_2)s_7+2m_2s_3}{m_1+m_2}&-U^{\prime\prime}(\frac14+s_5^2)-2U^{\prime}(\frac14+s_5^2)&\frac{2s_5U^{\prime\prime}(\frac14+s_5^2)(m_2-m_1)}{m_1+m_2}&0&s_3\\
s_3s_5&2U^{\prime\prime}(\frac14+s_5^2)s_5&-4U^{\prime\prime}(\frac14+s_5^2)s_5^2-2U^{\prime}(\frac14+s_5^2)&-s_3&0
\end{pmatrix},
\]
further substituting $s_7=-s_3/2$, and finally (in the characteristic polynomial) replacing $s_3^2$ by $2U^\prime\left(\frac14+s_5^2\right)$.\\
The eigenvalue $0$ must occur due to the first integral $j$; its eigenspace is transversal to the level sets of $j$. This argument proves linear stability.
\end{proof}
\begin{remark} \label{s4rem}{\em 
The condition $\eta_4=-\frac12$ also admits a natural physical interpretation. Indeed, with $\eta_4=\left<u,z\right>$ and $\left<u,u\right>=1$ one sees
\[
\begin{array}{cccl}
 & \eta_4&=&-\frac12\\
\Leftrightarrow&\left<u,2z+u\right>&=&0\\
\Leftrightarrow&\left<x-z,x+z\right>&=&0\\
\Leftrightarrow&\left<x,x\right>&=&\left<z,z\right>\\
\end{array}
\]
Thus the relative equilibria of this type are distinguished by the property that both particles have the same distance from the center all the time, with the barbell rotating around the center.}
\end{remark}
Both zeros of $\widehat\chi$ are negative whenever $C_1>0$ and $C_2>0$. This implies:
\begin{corollary}
Whenever $s_5\not=0$, and $U^\prime$ and $U^{\prime\prime}$ are positive functions then all stationary points of \eqref{reqmotion} are linearly stable.
\end{corollary}
As a counterpoint we consider the two dimensional gravitational potential (after a suitable scaling).
\begin{example}{\em 
When $U^\prime(r)=1/r$ then
\[
C_2=-\frac{16s_5^2}{(1/4+s_5^2)^3}<0;
\]
hence $\widehat\chi$ has a positive real root, and every stationary point with $s_5\not=0$ is unstable.}
\end{example}
\subsection{The special case of equal masses}\label{emsubec}
Here we take up the discussion of the further equilibria from subsection \ref{subsecone} in case $m_1=m_2$, with
\[
s_4=-\frac{m_1}{m_1+m_2}=-\frac12.
\]
In addition we have 
\[
s_5=s_6=0,\quad s_7=-\frac{s_3}2
\]
with $s_3$ arbitrary. The physical interpretation (using Remarks \ref{s5rem} and \ref{s4rem}) is straightforward: Both particles are on a straight line through the center,  on opposite sides with the same distance from the center. They rotate around the center, and the value of $s_3$, in view of $\eta_3=\det(v,u)$, determines the corresponding angular velocity and the angular momentum of the system. Linear orbital stability of the relative equilibria is determined by the roots of the quadratic polynomial
\[
\begin{array}{rcl}
\widehat\chi(\tau)&:=&\tau^2+(2s_3^2+U^{\prime\prime}(1/4)+4U^\prime(1/4))\tau\\
&&+s_3^4-s_3^2\left(U^{\prime\prime}(1/4)+4U^\prime(1/4)\right)+2U^{\prime\prime}(1/4)U^\prime(1/4)+4U^\prime(1/4)^2
\end{array}
\]
with discriminant
\[
D=8s_3^2\left(U^{\prime\prime}(1/4)+4U^\prime(1/4)\right)+U^{\prime\prime}(1/4)^2.
\]
Note that the discriminant becomes negative with increasing $s_3$ in case $U^{\prime\prime}(1/4)+4U^\prime(1/4)<0$; this implies instability. Whenever  $U^{\prime\prime}(1/4)+4U^\prime(1/4)\geq 0$ and $U^{\prime\prime}(1/4)\not=0$ one obtains stability for sufficiently large $s_3^2$ (i.e. sufficiently high veloities), by arguments similar to those used above.
\begin{example}{\em
Here we also look at the special case of gravitation in two dimensions; i.e., $U(r)=1/r$. One obtains
\[
\widehat\chi(\tau)=\tau^2+2s_3^2\tau+(s_3^2+8)(s_3^2-8)
\]
with constant discriminant $D=256$. For $s_3^2<8$ there is a positive root of $\widehat\chi$, and the corresponding relative equilibria are
unstable; for $s_3^2>8$ one has linear orbital stability. Roughly speaking, a
sufficiently fast rotation of the system around the center is linearly
orbitally stable in the present context.}
\end{example}
\section{Concluding remarks}
The example discussed in the present paper was chosen specifically to provide a nontrivial but easily manageable  illustration of the reduction method introduced in \cite{SchrWa}, and its adaptation to Hamiltonian systems. Thus, one criterion for the choice was that the polynomial invariant algebra of the group action does not admit a convenient (``small '') generator set, so that the use of localizations is necessary. On the other hand, the system was chosen with a view on computational convenience. \\
Compared to physically inspired and geometrically motivated reductions, the approach presented here is motivated by an emphasis on explicit computations.\\
A more involved application will be discussed in a forthcoming paper on the double spherical pendulum.\\
Essentially the same reduction procedure works for any Hamiltonian system in $\mathbb R^n$ that admits an $s$--dimensional toral symmetry group (compatible with the Poisson structure). This allows reduction to an $n-s$--dimensional system that admits $s$ independent first integrals; cf.~Schroeders \cite{SchrDiss}. Proposition \ref{redpropo} and its proof can readily be modified for this setting. The discussion of reduction with respect to non-abelian groups is a bit more involved, since suitable localizations are less easy to determine.\\

\noindent {\bf Acknowledgement.} The authors gratefully acknowledge support by the Research in Pairs program of MFO (Mathematisches Forschungsinstitut Oberwolfach) in July and August 2017.


\begin{thebibliography}{99}

\bibitem{ArnoldGren} V.I.~Arnold: {\it Sur la ge\'ometrie differentielle des groupes de Lie de dimension infinie et ses applications a l' hydrodynamique des fluids parfaits.} Ann.~Inst.~Fourier (Grenoble) {\bf 16}, 319 - 361 (1966).
\bibitem{ArnoldUsp} V.I.~Arnold: {\it The Hamiltonian nature of the Euler equations in the dynamics of a rigid body and of an ideal fluid} [Russian]. Usp. Mat. Nauk. {\bf 24}, 225 - 226 (1969).
\bibitem{ArnoldMech} V.I.~Arnold: {\it Mathematical Methods of Classical Mechanics.} Springer, New York (1978)
\bibitem{ArnoldGM} V.I.~Arnold: {\it Geometrical methods in the theory of ordinary differential equations.} Second Edition. Springer, New York (1988).


\bibitem{Bre} G.E.~Bredon: {\it Introduction to compact transformation groups.} Academic Press, New York (1972).

\bibitem{Cho} P.~Chossat: {\it The reduction of equivariant dynamics to the orbit space of compact group actions.} Acta Appl. Math. {\bf 70}, 71 - 94  (2002).

\bibitem{CuBa}  R.H.~Cushman, L.~Bates:
{\it Global Aspects of Classical Integrable Systems.} 2nd ed. Birkh\"auser, Boston (2015).



\bibitem{Fie} M.J.~Field: {\it Equivariant dynamical systems.} Trans. Amer. Math. Soc. {\bf 259}, 185 - 205 (1980).

\bibitem{Gros} F.D.~Grosshans: {\it Localization and invariant theory.} Advances in Mathematics {\bf 21}, 50-60 (1976).



\bibitem{GSW} F.D.~Grosshans, J.~Scheurle, S.~Walcher:
{\it Invariant sets forced by symmetry.}  J. Geom. Mechanics  {\bf 4}, 281 - 296  (2012).

\bibitem{Kummer}M.~Kummer: {\it On the construction of the reduced phase space of a Hamiltonian system with symmetry.} Indiana Univ.~Math.~J. {\bf 30}, 281 - 292 (1981).

\bibitem{Kus} A.G.~Kushnirenko: {\it An analytic action of a semisimple Lie group in a neighborhood of a fixed point is equivalent to a linear one.} Funct. Anal. Appl. {\bf 1}, 273 - 274  (1967).


\bibitem{Lun} D.~Luna: {\it Fonctions differentibles invariantes sous l'operation d'un groupe reductif.} Ann. Inst. Fourier (Grenoble) {\bf 26}, 33 - 49  (1976).
\bibitem{Marsden}J.~Marsden: {\it Lectures on mechanics.} London Math.~Soc.~Lecture Notes {\bf 174}, Cambridge Univ.~Press (1992).
\bibitem{MarRat} J.~Marsden, T.~Ratiu: {\it An introduction to mechanics and symmetry.} Springer, New York (1994).
\bibitem{MarWei}J.~Marsden, A.~Weinstein: {\it Reduction of symplectic manifolds with symmetry.} Rep. Math. Phys. {\bf 5}, 121 - 130 (1974).

\bibitem{Poe} V. Po\'enaru: {\it Singularit\'es $C^{\infty }$ en pr\'esence de sym\'etrie.} Lecture Notes in Mathematics {\bf 510}. Springer-Verlag, Berlin-New York (1976)

\bibitem{ProcesiSchwarz} C.~Procesi, G.~Schwarz: {\it  Inequalities defining orbit spaces.} Invent. Math. {\bf 81}, 539-554  (1985).

\bibitem{SSW} M.~Santoprete, J.~Scheurle, S.~Walcher: {\it Motion in a symmetric potential on the hyperbolic plane.} Canad. J. Math. {\bf 67}, 450 - 480 (2015).

\bibitem{Sche} J.~Scheurle: {\it Some aspects of successive bifurcations in the Couette-Taylor problem.} In J.~Chadam, M.~Golubitsky, W.F.~Langford, B.~Wetton: {\it Pattern formation: Symmetry methods and applications.} Fields Inst. Comm.~{\bf 5}, 335 - 345 (1996).

\bibitem{Schw} G.W.~Schwarz: {\it Smooth functions invariant under the action of a compact Lie group.} Topology {\bf 14}, 63 - 68 (1975).



\bibitem{SchrDiss} R.~Schroeders: {\it Symmetriereduktion polynomieller Vektorfelder.} Doctoral thesis, RWTH Aachen, to be submitted (2018).

\bibitem{SchrWa} R.~Schroeders, S.~Walcher: {\it Orbit space reduction and localizations.} Indag. Math. {\bf 27}, 1265 - 1278 (2016). 

\bibitem{WKM}    Li-Sheng Wang, P.S.~Krishnaprasad, J.H.~Maddocks: {\it Hamiltonian dynamics of a rigid body in a central gravitational field.} Celestial Mech.~and Dynamical Astronomy {\bf 50}, 349 - 386 (1990).


\end{thebibliography}
\end{document}